\documentclass[a4paper, reqno]{amsart}
\usepackage{amsaddr}
\usepackage[utf8]{inputenc}
\usepackage{hyperref}

\usepackage{amsthm, amsmath, amssymb}
\usepackage{xcolor}
\usepackage{enumitem}

\usepackage{pgfplots}
\pgfplotsset{compat = newest}

\usepackage{caption}

\newtheorem{theorem}{Theorem}
\newtheorem{corollary}[theorem]{Corollary}
\newtheorem{lemma}[theorem]{Lemma}
\newtheorem{proposition}[theorem]{Proposition}

\theoremstyle{definition}
\newtheorem{example}[theorem]{Example}

\newcommand{\R}{\mathbb{R}}
\newcommand{\N}{\mathbb{N}}

\newcommand{\func}{f}
\newcommand{\scaledfunc}{F}
\newcommand{\funcmin}{\func_{\min}}
\newcommand{\funcmax}{\func_{\max}}
\newcommand{\scaledfuncmin}{\scaledfunc_{\min}}
\newcommand{\scaledfuncmax}{\scaledfunc_{\max}}

\newcommand{\polys}[1]{\R[{#1}]}
\newcommand{\sos}[1]{\Sigma[{#1}]}

\newcommand{\positivecone}[1]{\mathcal{P}(\cube{{#1}})}
\newcommand{\richquadmodule}[1]{{Q}(\cube{n})_{#1}}

\newcommand{\cube}[1]{{\rm B}^{#1}}
\newcommand{\semialge}[1]{g_{#1}}

\newcommand{\richlowbound}[1]{{\func_{({#1})}}}

\newcommand{\infnorm}[1]{\| {#1} \|_{\infty}}
\newcommand{\Chebyinner}[2]{\langle {#1}, {#2} \rangle_\mu}
\newcommand{\Chebynorm}[1]{\| {#1} \|_{\mu}}

\newcommand{\Cheby}[1]{T_{#1}}
\newcommand{\Kernel}{\mathbf{K}}
\newcommand{\jacksoncoef}[2]{\lambda_{#1}^{#2}}
\newcommand{\jackpoly}{K^{\rm ja}}
\newcommand{\multijackpoly}{K}
\newcommand{\jackkernel}{\mathbf{K}^{\rm ja}}

\newcommand{\x}{\mathbf{x}}
\newcommand{\y}{\mathbf{y}}

\title[Schm\"udgen's Positivstellensatz for the hypercube]{An effective version of Schm\"udgen's Positivstellensatz for the hypercube}
\author{Monique Laurent}
\address{Centrum Wiskunde \& Informatica (CWI), Amsterdam and Tilburg University}
\email{monique.laurent@cwi.nl}
\author{Lucas Slot}
\address{Centrum Wiskunde \& Informatica (CWI), Amsterdam}
\email{lucas.slot@cwi.nl}
\subjclass[2010]{90C22, 90C23, 90C26}
\date{\today}

\begin{document}

\begin{abstract}
\small
Let $S \subseteq \R^n$ be a compact semialgebraic set and let $f$ be a polynomial nonnegative on $S$. Schm\"udgen's Positivstellensatz then states that for any $\eta > 0$, the nonnegativity of $f + \eta$ on $S$ can be certified by expressing $f + \eta$ as a conic combination of products of the polynomials that occur in the inequalities defining $S$, where the coefficients are  (globally nonnegative) sum-of-squares polynomials.
It does not, however, provide explicit  bounds on the degree of the polynomials required for such an expression. 
We show that in the special case where $S = [-1, 1]^n$ is the hypercube, a Schm\"udgen-type certificate of nonnegativity exists involving only polynomials of degree $O(1 / \sqrt{\eta})$. This improves quadratically upon the previously best known estimate in $O(1/\eta)$. Our proof relies on an application of the polynomial kernel method, making use in particular of the Jackson kernel on the interval $[-1, 1]$.
\end{abstract}

\maketitle

\section{Introduction}
Consider the problem of computing the global minimum:
\begin{equation}
    \label{EQ:mainproblem}
    \funcmin := \min_{\x \in \cube{n}} f(\x)
\end{equation}
of a polynomial $f$ of degree $d \in \N$ over the hypercube $\cube{n} := [-1,1]^n \subseteq \R^n$.
The program \eqref{EQ:mainproblem} can be reformulated as finding the largest $\lambda \in \R$ for which the function $f - \lambda$ is nonnegative on $\cube{n}$. That is, writing $\positivecone{n} \subseteq \polys{x}$ for the cone of all polynomials that are nonnegative on $\cube{n}$, we have:
\begin{equation}
    \label{EQ:nonnegativereformulation}
    \funcmin = \max \{ \lambda \in \R : f - \lambda \in \positivecone{n} \}.
\end{equation}
By replacing $\positivecone{n}$ in  \eqref{EQ:nonnegativereformulation} by a smaller subset of $\R[\x]$ one may obtain lower bounds on $\funcmin$.
One way of obtaining such subsets is based on the following description of $\cube{n}$ as a semialgebraic set:
\begin{equation} \label{EQ:semialgdesc}
    \cube{n} = \{ \x \in \R^n : \semialge{i}(\x) := (1-x_i^2) \geq 0 \quad \forall i \in [n] \}.
\end{equation}
In light of this description, we see that the \emph{preordering} $\richquadmodule{r}$, truncated at degree~$ r$,  defined by\footnote{{Sometimes the index $r$ is used in the literature to denote the truncation where all summands have degree at most $2r$. For our treatment here it is more convenient to let $r$ denote the truncation where all summands have degree at most $r$, the main reason being our use later of Theorem \ref{THM:MarkovLukacz}.}}:
\begin{equation}\label{eqmodule}
\richquadmodule{r} := \{\sum_{J \subseteq [ n]} \sigma_J \semialge{J} : \sigma_J \in \sos{\x}, ~\deg(\sigma_J g_J) \leq  r \} \quad (\semialge{J} := \prod_{j \in J} \semialge{j}),
\end{equation}
satisfies $\richquadmodule{r} \subseteq \positivecone{n}$ for all $r \in \N$. Here, $\sos{\x}$ is the set of sum-of-squares polynomials (i.e., of the form $p = p_1^2 + p_2^2 + \ldots + p_m^2$ for certain $p_i \in \R[\x]$). When no degree bounds are imposed (i.e., $r=\infty$) we obtain the full preordering $\richquadmodule{}$ generated by the polynomials $g_i(\x)=1-x_i^2$ ($i\in [n])$, which coincides with the quadratic module generated by the products $\prod_{i\in I}g_i(\x)$ ($I\subseteq [n]$). We thus obtain the following hierarchy of lower bounds for $\funcmin$, due to Lasserre \cite{Lasserre2001}:
\begin{equation}\label{EQ:lasr}
    \richlowbound{r} := \max \{ \lambda \in \R : f - \lambda \in \richquadmodule{r} \}. 
\end{equation}
If the program \eqref{EQ:lasr} is feasible, its maximum is attained. By definition, we have $\funcmin \geq \richlowbound{r + 1} \geq \richlowbound{r}$ for all $r \in \N$. Furthermore, we have $\lim_{r \to \infty} \richlowbound{r} = \funcmin$, which follows directly from the following special case of \emph{Schm\"udgen's Positvstellensatz}.

\begin{theorem}[Special case of Schm\"udgen's Positivstellensatz \cite{Schmudgen1991}]
\label{THM:positvstellensatz}
Let $f \in \positivecone{n}$ be a polynomial. Then for any $\eta > 0$ there exists an $r \in \N$ such that $f + \eta \in \richquadmodule{r}$.
\end{theorem}

\subsection{Main result}
We show a bound on the convergence rate of the lower bounds $\richlowbound{r}$ to the global minimum $\funcmin$ of $\func$ over $\cube{n}$ in $O(1/r^2)$. Alternatively, our result can be interpreted as a bound on the degree $r$ in Schm\"udgen's Positivstellensatz of the order $O(1/\sqrt{\eta})$ of a positivity certificate for $f+\eta$ when $f \in \positivecone{n}$.
\begin{theorem}
\label{THM:main}
Let $f$ be a polynomial of degree $d \in \N$. Then there exists a constant $C(n, d) > 0$, depending only on $n$ and $d$, such that:
\begin{equation}\label{eq:main}
\funcmin - \richlowbound{{(r+1)}n} \leq \frac{C(n, d)}{r^2} \cdot (\funcmax - \funcmin)\quad \text{ for all } r \geq \pi d \sqrt{2n}.
\end{equation}
Furthermore, the constant $C(n, d)$ may be chosen such that it either depends polynomially on $n$ (for fixed $d$) or it depends polynomially on $d$ (for fixed $n$), see relation (\ref{eqCnd}) for details.
\end{theorem}

\begin{corollary}
Let $f\in \positivecone{n}$ with degree $d$.
Then, for any $\eta>0$, we have:
$$f+\eta \in \richquadmodule{{(r+1)}n} \quad \text{ for all  } r \ge \max\Big\{ \pi d \sqrt{2n},{1\over \sqrt \eta} \sqrt{C(n,d) (\funcmax - \funcmin)}\Big\},$$
where $C(n,d)$ is the constant from Theorem \ref{THM:main}.
Hence we have $f+\eta\in \richquadmodule{r}$ for $r=O(1/\sqrt \eta)$.
\end{corollary}

\begin{proof}
Let $\eta>0$ and set $C_f:= C(n,d) \cdot (\funcmax - \funcmin)$.
 Pick an integer $r \ge \max\{ \pi d \sqrt{2n}, \sqrt{C_f/\eta}\}$.
Then we have:
$$f+\eta= \underbrace{f- \richlowbound{{(r+1)}n}}_{\in \richquadmodule{{(r+1)}n}} + \big(\underbrace{\richlowbound{{(r+1)}n}-\funcmin +{C_f\over r^2}}_{\ge 0 \text{ by Theorem \ref{THM:main}}}\big)
+\underbrace{ \funcmin}_{\ge 0 } + \big( \underbrace{\eta - {C_f\over r^2}}_{\ge 0}\big),
$$
which shows $f+\eta\in \richquadmodule{{(r+1)}n}$.
\end{proof}

\subsection{Outline of the proof}
\label{SEC:outline}
Let $f \in \R[\x]$ be a polynomial of degree $d$. 
To simplify our arguments and notation, we will work with the scaled function:
\[
    \scaledfunc := \frac{\func - \funcmin}{\funcmax - \funcmin},
\]
for which $\scaledfuncmin = 0$ and $\scaledfuncmax = 1$. Since the inequality (\ref{eq:main}) is invariant under a positive scaling of $f$ and adding a constant, it indeed suffices to show the result for the function $F$.

The idea of the proof is as follows. Let $\epsilon > 0$ and consider the polynomial $\tilde \scaledfunc := \scaledfunc + \epsilon$. {Let $r\ge d$}. Suppose that we are able to construct a (nonsingular) linear operator $\Kernel_r : \polys{\x}_r \rightarrow \polys{\x}_r$
which has the following two properties:
\begin{align}
	\label{PROPERTY:incone}
	&\Kernel_r p \in \richquadmodule{{(r+1)}n} \quad \text{ for all } p \in \positivecone{n}_{ r}, \tag{P1} \\
	\label{PROPERTY:infnorm}
    &\infnorm{\Kernel_r^{-1} \tilde \scaledfunc - \tilde \scaledfunc} := \max_{\x \in \cube{n}} |\Kernel_r^{-1} \tilde \scaledfunc(\x) - \tilde \scaledfunc(\x)| \leq \epsilon \tag{P2}.
\end{align}
Then, by \eqref{PROPERTY:infnorm}, we have $\Kernel_r^{-1} \tilde \scaledfunc \in \positivecone{n}_{ r}$. Indeed, as $F$ is nonnegative on $\cube{n}$,  $\tilde F(\x) = F(\x) + \epsilon$ is greater than or equal to $\epsilon$ for all $\x \in \cube{n}$, and so \eqref{PROPERTY:infnorm} tells us that after application of the operator $\Kernel_r^{-1}$, the resulting polynomial $\Kernel_r^{-1} \tilde F$ is nonnegative on $\cube{n}$. Using \eqref{PROPERTY:incone}, we may then conclude that $\tilde \scaledfunc = \Kernel_r (\Kernel_r^{-1} \tilde \scaledfunc) \in \richquadmodule{{(r+1)}n}$. 
It follows that 
$-\epsilon \le F_{({(r+1)}n)}$, i.e., $F_{\min}-F_{({(r+1)}n)}  \le \epsilon$, and thus $\funcmin - \richlowbound{{(r+1)}n} \leq \epsilon 
\cdot (\funcmax - \funcmin)$. We collect this in the next lemma for future reference.

\begin{lemma}
\label{LEM:outline}
Assume that for some $r \ge d$ and $\epsilon > 0$ 
there exists a nonsingular operator $\Kernel_r : \polys{\x}_r \rightarrow \polys{\x}_r$ which satisfies the properties \eqref{PROPERTY:incone} and \eqref{PROPERTY:infnorm}. Then we have
\[
	\funcmin - \richlowbound{{(r+1)}n} \leq \epsilon \cdot (\funcmax - \funcmin).
\]
\end{lemma}
In what follows, we will construct such an operator $\Kernel_r$ for each $r \geq \pi d \sqrt{2n}$ and the parameter $\epsilon := C(n,d) / r^2$, where the constant $C(n,d)$ will be specified later. Our main Theorem \ref{THM:main} then follows after applying Lemma \ref{LEM:outline}. 

We make use of the \emph{polynomial kernel method} for our construction: after choosing a suitable kernel ${K_r:\R^n\times\R^n\to\R}$, we define the linear operator ${\Kernel_r:\R[\x]_r\to\R[\x]_r}$ via the integral transform:
\[
	\Kernel_r p(\x) := \int_{\cube{n}} K_r(\x, \y) p(\y) d\mu(\y) \quad (p \in \R[\x]_r).
\]
Here, $\mu$ is the \emph{Chebyshev measure}  on $\cube{n}$ as defined in \eqref{EQ:chebymeasure} below.
A good choice for the kernel $K_r$ is a multivariate version (see Section \ref{SEC:construction}) of the well-known \emph{Jackson kernel} $\jackpoly_r$ of degree $r$ (see Section \ref{SEC:jackson}). For this choice of kernel, the operator $\Kernel_r$ naturally satisfies \eqref{PROPERTY:incone} (see Section \ref{SEC:p1}). Furthermore, it diagonalizes with respect to the basis of $\R[\x]$ given by the (multivariate) \emph{Chebyshev polynomials} (see Section \ref{SEC:cheby}). Property \eqref{PROPERTY:infnorm} can then be verified by analyzing the eigenvalues of $\Kernel_r$, which are closely related to the expansion of $\jackpoly_r$ in the basis of (univariate) Chebyshev polynomials (see Section \ref{SEC:p2}). We end this section by illustrating our method of proof with a small example.

\begin{example}
\label{EX:1}
Consider the polynomial $f(x) = 1 - x^2 - x^3 + x^4$, which is nonnegative on $[-1,1]$. For $r \in \N$, let $\Kernel_r$ be the operator associated to the univariate Jackson kernel \eqref{EQ:univariatejackpoly} of degree $r$, which satisfies \eqref{PROPERTY:incone} (see Section \ref{SEC:p1}). For $\eta = 0.1$, we observe that applying $\Kernel_7^{-1}$ to $f + \eta$ yields a nonnegative function on $[-1, 1]$, whereas applying $\Kernel_5^{-1}$ does not (see Figure \ref{FIG:example}). Applying the arguments of Section \ref{SEC:outline}, we may thus conclude that $f + \eta \in \richquadmodule{{8}}$, but not that  $f + \eta \in \richquadmodule{{6}}$.
\end{example}
\begin{figure}
\centering
\begin{tikzpicture}
    \begin{axis}[
    xmin = -1, xmax = 1,
    ymin = -0.5, ymax = 3.0, grid = both,
    xtick = {-1, -0.5, 0, 0.5, 1},
    width=0.9\textwidth,
    height=18em,
    legend style={font=\small}
    ]
        \addplot[domain = -1:1, samples = 200, smooth, ultra thick] 
        {1 + 0*x -1*x^2 -1*x^3 + 1*x^4 + 0.1};
        \addplot[domain = -1:1, samples = 200, smooth,  ultra thick, orange, dashed] 
        {1.51985 + 0.968875*x -5.15883*x^2 -2.40175*x^3 + 5.15883*x^4 + 0.1};
        \addplot[domain = -1:1, samples = 200, smooth, ultra thick, red, dashdotted] 
        {1.18978 + 0.456657*x -2.5182*x^2 -1.67305*x^3 + 2.5182*x^4 + 0.1};
		\legend{$f(x) + \eta$,$\Kernel_5^{-1}\big(f(x) + \eta\big)$,$\Kernel_7^{-1}\big(f(x) + \eta\big)$}
    \end{axis}
\end{tikzpicture}
\captionsetup{width=0.8\textwidth, font=small}
\caption{The polynomial $f(x) + \eta$ of Example \ref{EX:1} and its transformations under the inverse operators $\Kernel_5^{-1}$ and $\Kernel_7^{-1}$ associated to the Jackson kernels of degree $5$ and $7$.}
\label{FIG:example}
\end{figure}
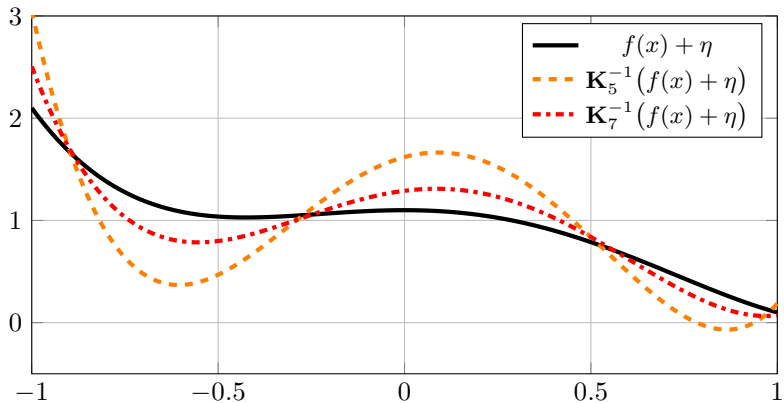

\subsection{Related work}
The polynomial kernel method, which forms the basis of our analysis, is widely used in functional approximation, see, e.g., \cite{KernelPolynomialMethodSurvey}. In the present context, the method has already been employed for the analysis of the sum-of-squares hierarchy for optimization over the hypersphere $S^{n-1}$ in \cite{FangFawzi2021} (where a rate in $O(1/r^2)$ was shown as well)  and for optimization over the \emph{binary} cube $\{-1, 1\}^n$ in \cite{SlotLaurent2020b}. There, the authors use kernels that are invariant under the symmetry of $S^{n-1}$ and $\{-1, 1\}^n$, respectively.

 In \cite{HessdeKlerkLaurent2017}, the polynomial kernel method, and the Jackson kernel in particular, were used to analyze the quality of a related Lasserre-type hierarchy of \emph{upper} bounds on $\funcmin$ over $\cube{n}=[-1, 1]^n$, where one searches for 
 a density {in the truncated preordering $\richquadmodule{r}$ }
  minimizing the expected value of $f$ over $\cube{n}$ (showing again a convergence rate in $O(1/r^2)$).

For a general compact semialgebraic set $S$, a polynomial $f$ nonnegative on $S$ and $\eta>0$, existence of Schm\"udgen-type certificates of positivity for $f + \eta$ with degree bounds in $O(1/\eta^c)$ was shown in \cite{Schweighofer2004}, 
 where $c > 0$ is a constant depending on $S$. 
This result uses different tools, including in particular a representation result for polynomial optimization over the simplex by P\'olya \cite{Polya} and the effective degree bounds by Powers and Reznick \cite{PR}.

For the case of the hypercube\footnote{The hypercube $[0,1]^n$ is considered in \cite{deKlerkLaurent2010} but the results extend to the hypercube $[-1,1]^n$ by an affine change of variables.} a degree bound  in $O(1/\eta)$ for Schm\"udgen-type certificates is obtained in \cite{deKlerkLaurent2010}, thus showing that one can take $c\le 1$ in the above mentioned result of \cite{Schweighofer2004}.
This result holds in fact for a weaker hierarchy of bounds  obtained by restricting in (\ref{EQ:lasr}) to decompositions of the polynomial $f-\lambda$ involving factors $\sigma_J$ that are nonnegative scalars (instead of sums of squares), also known as Handelman-type decompositions (thus replacing the preordering $ \richquadmodule{r}$ by its subset $H_r$ of  polynomials having a Handelman-type decomposition).
The analysis in \cite{deKlerkLaurent2010} relies on employing the \emph{Bernstein operator} $\mathbf{B}_r$, which has the property of mapping a polynomial  nonnegative over the hypercube to a polynomial in the set $H_{rn}\subseteq  \richquadmodule{rn}$.

In this paper, we can show a further improvement by using a different type of kernel operator; namely we show that we can take the constant  $c \leq  1/2$ in the special case $S = [-1, 1]^n$. 

\section{Preliminaries}
\subsection{Notation}
Throughout, $\cube{n} := [-1, 1]^n \subseteq \R^n$ is the $n$-dimensional hypercube. We write $\R[x]$ for the univariate polynomial ring, while reserving the bold-face notation $\R[\x] = \R[x_1, x_2, \dots, x_n]$ to denote the ring of polynomials in $n$ variables. Similarly, $\Sigma[x] \subseteq \R[x]$ and $\Sigma[\x] \subseteq \R[\x]$ denote the sets of univariate and $n$-variate sum-of-squares polynomials, respectively, consisting of all polynomials of the form $p = p_1^2 + p_2^2 + \dots + p_m^2$ for certain polynomials $p_1,\ldots,p_m$ and $m\in \N$. For a polynomial $p \in \R[\x]$, we write $p_{\min}, p_{\max}$ for its minimum and maximum over $\cube{n}$, respectively, and $\infnorm{p} := \sup_{\x \in \cube{n}} |p(\x)|$ for its sup-norm on $\cube{n}$.
\subsection{Chebyshev polynomials}
\label{SEC:cheby}
Let $\mu$ be the normalized \emph{Chebyshev measure} on $\cube{n} = [-1, 1]^n$, defined by:
\begin{equation}
\label{EQ:chebymeasure}
d\mu(\x) = \frac{dx_1}{\pi \sqrt{1 - x_1^2}} \ldots \frac{dx_n}{\pi \sqrt{1 - x_n^2}}.
\end{equation}
Note that $\mu$ is a probability measure on $\cube{n}$, meaning that $\int_{\cube{n}}d\mu = 1$.
 We write $\Chebyinner{\cdot}{\cdot}$ for the corresponding inner product on $\R[\x]$, given by:
\[
	\Chebyinner{f}{g} := \int_{\cube{n}} f(\x) g(\x) d\mu(\x).
\]
For $k \in \N$, let $\Cheby{k}$ be the univariate \emph{Chebyshev polynomial} (see, e.g., \cite{Szego1959}) of degree $k$, defined by:
\[
	\Cheby{k}(\cos \theta) := \cos (k\theta) \quad (\theta \in \R).
\]
Note that $|T_k(x)| \leq 1$ for all $x \in [-1, 1]$ and that $T_0 = 1$.
The Chebyshev polynomials satisfy the orthogonality relations:
\begin{equation}
\label{EQ:orthorelations}
\Chebyinner{T_a}{T_b} = \int_{-1}^{1} \Cheby{a}(x) \Cheby{b}(x) d\mu(x) = 
\begin{cases} 0 \quad a \neq b, \\ 1 \quad a=b=0, \\ \frac{1}{2} \quad a=b\neq 0. \end{cases}
\end{equation} A univariate polynomial $p$ may therefore be expanded as:
\[
p = p_0 + \sum_{k = 1}^{\deg(p)} 2 p_k T_k, \quad \text{where } p_k := \Chebyinner{T_k}{p}.
\]
For $\kappa \in \N^n$, we consider the \emph{multivariate} Chebyshev polynomial $\Cheby{\kappa}$, defined by setting:
\[
\Cheby{\kappa}(\x) := \prod_{i = 1}^n \Cheby{\kappa_i}(x_i).
\]
The multivariate Chebyshev polynomials form a basis for $\R[\x]$ and satisfy the orthogonality relations:
\begin{equation}
\label{EQ:orthorelationsmulti}
\Chebyinner{T_\alpha}{T_\beta} = \int_{\cube{n}} \Cheby{\alpha}(\x) \Cheby{\beta}(\x) d\mu(\x)  = \begin{cases} 0 \quad &\alpha \neq \beta, \\ 1 \quad &\alpha=\beta=0, \\ 2^{-w(\alpha)} \quad &\alpha=\beta\neq 0. \end{cases}
\end{equation}
Here, $w(\alpha) := |\{i \in [n] : \alpha_i \neq 0 \}|$ denotes the Hamming weight of $\alpha \in \N^n$.

We use the notation $\N^n_{d} \subseteq \N^n$ to denote the set of $n$-tuples $\alpha\in \N^n$ with $|\alpha|=\sum_{i=1}^n\alpha_i\le d$.
As in the univariate case, we may expand any $n$-variate polynomial $p$ as:
\begin{equation}
\label{EQ:multiChebyexp}
p = \sum_{\kappa \in \N^n_{\deg(p)}} 2^{w(\kappa)} p_\kappa T_\kappa, \quad\text{where } p_\kappa := \Chebyinner{T_\kappa}{p}.
\end{equation}

\subsection{The Jackson kernel}
\label{SEC:jackson}
For $r\in \N$ and for coefficients $\jacksoncoef{k}{r} \in \R$ to be specified below in (\ref{EQ:jacksoncoef}), consider the kernel $\jackpoly_r : \R \times \R \to \R$ given by:
\begin{equation}
    \label{EQ:univariatejackpoly}
    \jackpoly_r(x, y) := 1 + 2 \sum_{k=1}^r \jacksoncoef{k}{r} \Cheby{k}(x) \Cheby{k}(y).
\end{equation}
We associate a linear operator $\jackkernel_r : \polys{x}_r \to \polys{x}_r$ to this kernel by setting:
\[
\jackkernel_r p(x) := \int_{-1}^{1} \jackpoly_r(x, y) p(y) d\mu(y) \quad (p \in \R[x]_r).
\]
Using the orthogonality relations \eqref{EQ:orthorelations}, and writing $\lambda_0^r :=1$, we see that:
\[
\jackkernel_r \Cheby{k}(x) := \int_{-1}^1 \jackpoly_r(x, y) \Cheby{k}(y) d\mu(y) = \lambda^r_k \Cheby{k}(x) \quad (0 \leq k \leq r).
\]
In other words, $\jackkernel_r$ is a diagonal operator with respect to the Chebyshev basis of $\polys{x}_r$, and its eigenvalues are given by $\lambda^r_0 = 1, \jacksoncoef{1}{r}, \dots, \jacksoncoef{r}{r}$. 
In what follows, we set:
\begin{equation}
    \label{EQ:jacksoncoef}
    \jacksoncoef{k}{r} = \frac{1}{r+2}\big((r+2-k) \cos(k \theta_r) + \frac{\sin(k\theta_r)}{\sin(\theta_r)} \cos(\theta_r) \big) \quad (1 \leq k \leq r),
\end{equation}
with $\theta_r = \frac{\pi}{r+2}$. We then obtain the so-called \emph{Jackson kernel} (see, e.g., \cite{KernelPolynomialMethodSurvey}). The following properties of the Jackson kernel are crucial to our analysis.
\begin{proposition}
\label{PROP:jacksoncoefficients}
For every $d, r \in \N$ with $d \leq r$, we have:
\begin{enumerate}
	\item[(i)] $\jackpoly_r(x, y) \geq 0$ for all $x, y \in [-1, 1]$,
	\item[(ii)] $1 \geq \jacksoncoef{k}{r} > 0$ for all $0 \leq k \leq r$,  and
    \item[(iii)]  $|1 - \jacksoncoef{k}{r}| = 1 - \jacksoncoef{k}{r} \leq {\pi^2 d^2\over (r+2)^2}$ for all $0 \leq k \leq d$.
\end{enumerate}
\end{proposition}
\begin{proof}
Nonnegativity of the Jackson kernel is a well-known fact, and is verified, e.g., in \cite{HessdeKlerkLaurent2017}. We check that the other properties (ii)-(iii) hold as well.

\smallskip\noindent
\textbf{Second property (ii):} Note that when $k \leq (r+2) / 2$, both terms of \eqref{EQ:jacksoncoef} are positive, and so certainly $\jacksoncoef{k}{r} > 0$. 
So assume $(r+2) / 2 < k \leq r$. Set $h = r+2 - k$, so that $k \theta_r = \pi - h \theta_r$, {$2\le h<(r+2)/2$},  and 
\begin{equation}
\label{EQ:jacksoncoefinh}
    (r+2) \jacksoncoef{k}{r} = - h \cos( h \theta_r) + \frac{\sin (h \theta_r)}{\sin (\theta_r)}\cos(\theta_r).
\end{equation}
It remains to show that the RHS of \eqref{EQ:jacksoncoefinh} is positive for all $2 \leq h < (r+2)/2$.
Note that $1 > \cos(\theta_r) > 0$, $\sin(\theta_r)  > 0$ and that $\sin(h \theta_r) \geq 0$ for all $2 \leq h < (r+2)/2$. We proceed by induction on $h$. For $h = 2$, we compute:
\begin{equation}
\label{EQ:jacksoncoefIH}
\begin{split}
    - h \cos( h \theta_r) + \frac{\sin (h \theta_r)}{\sin(\theta_r)}\cos(\theta_r) 
    &= -2 (2\cos^2(\theta_r) -1) + 2 \cos^2(\theta_r) \\
    &= - 2 \cos^2(\theta_r) + 2 > 0,
\end{split}
\end{equation}
{which settles the base of induction.} For $h \geq 2$, we compute:
\begin{align*}
    -(h+1) &\cos( (h+1) \theta_r) + \sin((h+1)\theta_r)\frac{\cos(\theta_r)}{\sin(\theta_r)} \\
    &= -(h+1) \big(\cos(h \theta_r)\cos(\theta_r) - \sin(h \theta_r)\sin(\theta_r) \big) \\
    &\quad + \big(\sin(h\theta_r)\cos(\theta_r) + \cos(h\theta_r) \sin(\theta_r)\big) \frac{\cos(\theta_r)}{\sin(\theta_r)} \\
    &= - h \cos (h \theta_r) \cos(\theta_r) + (h+1) \sin (h\theta_r)\sin(\theta_r) + \frac{\sin (h \theta_r)}{\sin(\theta_r)}\cos^2(\theta_r) \\
    & {= \underbrace{\cos(\theta_r)}_{>0} \underbrace{\Big( -h \cos(h\theta_r) + {\sin(h\theta_r)\over \sin(\theta_r)}\cos(\theta_r)\Big)}_{\ge 0 \text{\rm\ by the induction assumption}} + (h+1)\underbrace{\sin(h\theta_r)\sin(\theta_r)}_{\ge 0}}\\
& \ge 0.
\end{align*}
We conclude that $\jacksoncoef{k}{r} > 0$ for all $k \in [r]$. 
To see that $\jacksoncoef{k}{r} \leq 1$, note that for all $k \in \N$, $\Cheby{k}(x) \leq 1$ for $-1 \leq x \leq 1$ and $\Cheby{k}(1) = 1$. We can thus compute:
\[
	\jacksoncoef{k}{r} = \jacksoncoef{k}{r} \Cheby{k}(1) =  \int_{-1}^1 \jackpoly_r(1, y) \Cheby{k}(y) d\mu(y) \leq \int_{-1}^1 \jackpoly_r(1, y) d\mu(y) = \jacksoncoef{0}{r} = 1,
\]
making use of the nonnegativity of $\jackpoly_r(x, y)$ on $[-1,1]^2$ for the inequality.

\smallskip\noindent
\textbf{Third property (iii):}
Using the expression of $\lambda^k_r$ in (\ref{EQ:jacksoncoef}) we have
$$1-\lambda^k_r= 1- {r+2-k\over r+2}\cos(k\theta_r) - {1\over r+2}{\sin(k\theta_r)\cos(\theta_r)\over \sin(\theta_r)}.$$
We now bound each trigonometric term using the fact that:
\begin{align}\label{eqineq}
\cos(x)\ge 1-{1\over 2}x^2,\quad x-{1\over 6}x^3\le \sin(x)\le x\quad (x \in \R).\end{align}
{ When $k=1$ we immediately get:
$$1-\lambda^1_r= 1-\cos(\theta_r) \le {1\over 2}\theta_r^2= {\pi^2\over 2 (r+2)^2} \le {d^2\pi^2\over (r+2)^2}.
$$}
{Assume now $2\le k\le d$. Using (\ref{eqineq}) combined with $\cos(\theta_r), \sin(\theta_r), \sin (k\theta_r)>0$ we obtain:}
$$
{\sin(k\theta_r)\cos(\theta_r)\over \sin(\theta_r)}
 \ge \big ( k\theta_r-{1\over 6}k^3\theta_r^3\big) \big(1-{1\over 2}\theta_r^2\big){1\over \theta_r}
 \ge k -{k\over 2}\theta_r^2\big(1+{k^2\over 3}\big )
$$ and thus:
\begin{align*}
1-\lambda^k_r  & \le 1-{r+2-k\over r+2}\big(1-{k^2\theta_r^2\over 2}\big) 
-{1\over r+2} \Big(k -{k\over 2}\theta_r^2\big(1+{k^2\over 3}\big )\Big)\\
&=\underbrace{{r+2-k\over r+2}}_{\le 1} {k^2\theta_r^2\over 2}
+ \underbrace{{k\over 2(r+2)}}_{\le 1/2} \theta_r^2 \underbrace{\big(1+{k^2\over 3}\big)}_{\le {2\over 3}k^2 \text{ if } k\ge 2}\\
& \le k^2\theta_r^2 \le {d^2\pi^2\over (r+2)^2}.
\end{align*}
This concludes the proof if $k\ge 2$. 
\end{proof}

\section{Proof of the main theorem}
\subsection{Construction of the linear operator  $\Kernel_r$}
\label{SEC:construction}
As noted before, in order to prove Theorem \ref{THM:main} it suffices to construct a linear operator $\Kernel_r : \polys{\x}_r \to \polys{\x}_r$ that is nonsingular and satisfies \eqref{PROPERTY:incone} and \eqref{PROPERTY:infnorm}. For this purpose we define the multivariate Jackson kernel $\multijackpoly_r : \R^n \times \R^n \to \R$ by setting:
\begin{equation}
	\label{EQ:multijackpoly}
	\multijackpoly_r(\x, \y) := \prod_{i = 1}^n \jackpoly_r(x_i, y_i),
\end{equation}
where $\jackpoly_r$ is the (univariate) Jackson kernel from $\eqref{EQ:univariatejackpoly}$.
Now let $\Kernel_r$ be the corresponding kernel operator defined by:
\[
\Kernel_r p(\x) = \int_{\y \in \cube{n}} \multijackpoly_r(\x,\y) p(\y) d \mu(\y) \quad (p \in \R[\x]_r).
\]
The operator $\Kernel_r$ is diagonal w.r.t. the (multivariate) Chebyshev basis, and its eigenvalues can be expressed in terms of the coefficients $\jacksoncoef{k}{r}$ of the Jackson kernel, as the following lemma shows.
\begin{lemma}
\label{LEM:KernelDiag}
The operator $\Kernel_r$ is diagonal w.r.t. the Chebyshev basis for $\polys{\x}_r$, and its eigenvalues are given by: 
\[
\jacksoncoef{\kappa}{r} := \prod_{i = 1}^n \jacksoncoef{\kappa_i}{r} \quad (\kappa \in \N^n_r).
\]
\end{lemma}
\begin{proof}
For $\kappa \in \N^n_r$, we see that:
\begin{align*}
\Kernel_r \Cheby{\kappa}(\x) &= \int_{\y \in \cube{n}} \multijackpoly_r(\x,\y) \Cheby{\kappa}(\y) d \mu(\y) \\
&= \prod_{i = 1}^n \bigg(\int_{y_i \in [-1, 1]}  \jackpoly_r(x_i, y_i) \Cheby{\kappa_i}(y_i) d \mu(y_i)\bigg)
=\prod_{i = 1}^n \jacksoncoef{\kappa_i}{r} \Cheby{\kappa_i}(x_i) = \jacksoncoef{\kappa}{r} \Cheby{\kappa}(\x), 
\end{align*}
as required.
\end{proof} \noindent
It follows immediately from Proposition \ref{PROP:jacksoncoefficients}(ii) that $\Kernel_r$ has only nonzero eigenvalues and thus is non-singular. We show that $\Kernel_r$ further satisfies \eqref{PROPERTY:incone} and \eqref{PROPERTY:infnorm}.
\subsection{Verification of property \eqref{PROPERTY:incone}}
\label{SEC:p1}
Consider the following strengthening of \\ Schm\"udgen's Positivstellensatz in the univariate case.
\begin{theorem}[Fekete, Markov-Luk\'acz (see \cite{PR2})]\label{THM:MarkovLukacz}
Let $p$ be a univariate polynomial of degree $r$, and assume that $p \geq 0$ on the interval $[-1, 1]$. Then $p$ admits a representation of the form:
\begin{equation}
\label{EQ:Feketerepresentation}
p(x) = \sigma_0(x) + \sigma_1(x) (1-x^2),
\end{equation}
where $\sigma_0, \sigma_1 \in \sos{x}$ and $\sigma_0$ and  $\sigma_1 \cdot (1-x^2)$ are of degree at most $r+1$. In other words, in view of (\ref{eqmodule}), we have $p\in Q([-1, 1])_{{r+1}}$.
\end{theorem}

By Proposition \ref{PROP:jacksoncoefficients}(i), for any $y \in [-1, 1]$, the polynomial $x \mapsto \jackpoly_r(x, y)$ is nonnegative on $[-1,1]$ and thus, by Theorem \ref{THM:MarkovLukacz}, it 
belongs to $Q([-1, 1])_{{r+1}}$. 
This implies directly that the multivariate polynomial $\x \mapsto \multijackpoly_r(\x, \y) = \prod_{i=1}^n \jackpoly_r(x_i, y_i)$ belongs to $\richquadmodule{{(r+1)}n}$ for all $\y \in [-1, 1]^n$.
\begin{lemma}
The operator $\Kernel_r$ satisfies property \eqref{PROPERTY:incone}, that is, we have $\Kernel_r p \in \richquadmodule{{(r+1)}n}$ for all $p \in \positivecone{n}_{{r}}$.
\end{lemma}
\begin{proof}
One way to see this is as follows. Let $\{\y_i: i\in [N]\} \subseteq \cube{n}$ and $w_i>0$ ($i\in [N]$) form a quadrature rule for  integration of degree $2r$ polynomials over $\cube{n}$; that is, $\int_{\cube{n}} p(\x)d\mu(\x)=\sum_{i=1}^N w_ip(\y_i)$ for any $p\in \R[\x]_{2r}$.
Then, for any $p\in  \positivecone{n}_r$, we have
$\Kernel_r p(\x)=\sum_{i=1}^N K_r(\x,\y_i)p(\y_i) w_i$ with $p(\y_i)w_i\ge 0$ for all $i$, which shows that $\Kernel_r p\in 
\richquadmodule{{(r+1)}n}$.
\end{proof}

\subsection{Verification of property \eqref{PROPERTY:infnorm}}
\label{SEC:p2}
We may decompose the polynomial ${\tilde F = F + \epsilon}$ into the multivariate Chebyshev basis \eqref{EQ:multiChebyexp}:
\[
	\tilde F = \epsilon + \sum_{\kappa \in \N^n_d} 2^{w(\kappa)} F_\kappa \Cheby{\kappa}, \quad \text{where } F_{\kappa} = \Chebyinner{F}{T_\kappa}.
\]
By Lemma \ref{LEM:KernelDiag}, we then have:
\begin{equation}
\label{EQ:infnormbound}
\begin{split}
\infnorm{\Kernel_r^{-1} \tilde \scaledfunc - \tilde \scaledfunc} 
&= \infnorm{\sum_{\kappa \in \N^n_d} (1/\jacksoncoef{\kappa}{r}) 2^{w(\kappa)} \scaledfunc_\kappa \Cheby{\kappa} - 2^{w(\kappa)}\scaledfunc_\kappa \Cheby{\kappa}} \\
&\leq \sum_{\kappa \in \N^n_d} 2^{w(\kappa)} |\scaledfunc_\kappa| |1-1/\jacksoncoef{\kappa}{r}|,
\end{split}
\end{equation}
making use of the fact that $\lambda_0 = 1$ and $|\Cheby{\kappa}(x)| \leq 1$ for all $x \in \cube{n}$. It remains to analyze the expression at the right-hand side of $\eqref{EQ:infnormbound}$. First, we bound the size of $|\scaledfunc_\kappa|$ for $\kappa \in \N^n$.
\begin{lemma}\label{lemkappa}
We have $|\scaledfunc_\kappa| = |\Chebyinner{\scaledfunc}{\Cheby{\kappa}}| \leq 2^{-w(\kappa) / 2}$ for all $\kappa \in \N^n$.
\end{lemma}
\begin{proof}
Since $\mu$ is a probability measure on $\cube{n}$, we have $\Chebynorm{\scaledfunc} \leq \infnorm{\scaledfunc} \leq 1$. Using the Cauchy-Schwarz inequality and \eqref{EQ:orthorelationsmulti}, we then find:
\[
	\Chebyinner{\scaledfunc}{\Cheby{\kappa}} \leq \Chebynorm{\scaledfunc_\kappa} \Chebynorm{\Cheby{\kappa}} \leq \Chebynorm{\Cheby{\kappa}} = 2^{-w(\kappa) / 2}.
\]
\end{proof}
\noindent
To bound the parameter $|1-1/\jacksoncoef{\kappa}{r}|$, we first prove a bound on $|1-\jacksoncoef{\kappa}{r}|$, which we obtain by applying Bernoulli's inequality.
\begin{lemma}[Bernoulli's inequality]
For any $x \in [0, 1]$ and $t \geq 1$, we have:
\begin{equation}
	\label{EQ:Bernoulli}
	1 - (1-x)^t \leq tx.
\end{equation}
\end{lemma}
\begin{lemma}\label{lemlambda}
For any $\kappa \in \N^n_d$ and $r \geq \pi d$, we have:
\[
	|1 - \jacksoncoef{\kappa}{r}| \leq \frac{n \pi^2 d^2}{r^2}.
\]
\end{lemma}
\begin{proof}
By Proposition \ref{PROP:jacksoncoefficients}, we know that $ 0 \leq \gamma_k := (1 - \jacksoncoef{k}{r}) \leq \pi^2d^2 / r^2 \leq 1$ for $0 \leq k \leq d$. Writing $\gamma := \max_{0 \leq k \leq d} \gamma_k$, we compute:
\[
	1 - \jacksoncoef{\kappa}{r} = 1 - \prod_{i=1}^n \jacksoncoef{\kappa_i}{r} = 1 - \prod_{i=1}^n (1 - \gamma_{\kappa_i}) \leq 1 - (1 - \gamma)^n \leq n \gamma \leq \frac{n \pi^2 d^2}{r^2},
\]
making use of \eqref{EQ:Bernoulli} for the second to last inequality.
\end{proof}
\begin{lemma} \label{LEM:rlarge}
Assuming that $r \geq \pi d \sqrt{2n}$, we have:
\[
	| 1 - 1/\jacksoncoef{\kappa}{r}| \leq \frac{2n \pi^2 d^2}{r^2}.
\]
\end{lemma}
\begin{proof}
Under the assumption, and using the previous lemma, we have $| 1 - \jacksoncoef{\kappa}{r} | \leq 1/2$, which implies that $\jacksoncoef{\kappa}{r} \geq 1/2$. We may then bound:
\[
|1 - 1/\jacksoncoef{\kappa}{r}| = | \frac{1 - \jacksoncoef{\kappa}{r}}{\jacksoncoef{\kappa}{r}}| \leq 2 |1 - \jacksoncoef{\kappa}{r}| \leq \frac{2n \pi^2 d^2}{r^2}.
\]
\end{proof}
\noindent
Putting things together and using (\ref{EQ:infnormbound}), Lemma \ref{lemkappa} and Lemma \ref{lemlambda}  we find that:
\begin{align*}
\infnorm{\Kernel_r^{-1} \tilde \scaledfunc - \tilde \scaledfunc} &\leq 
\sum_{\kappa \in \N^n_d} 2^{w(\kappa)} |\scaledfunc_\kappa| |1-1/\jacksoncoef{\kappa}{r}| \\
&\leq \sum_{\kappa \in \N^n_d} 2^{w(\kappa) / 2} \cdot \frac{2n \pi^2 d^2}{r^2} 
\leq |\N^n_d| \cdot \max_{\kappa \in \N^n_d} 2^{w(\kappa) / 2} \cdot \frac{2{n}\pi^2 d^2}{r^2}.
\end{align*}
Hence $\Kernel_r$ satisfies \eqref{PROPERTY:infnorm} with $\epsilon=C(n,d)/r^2$, where:
\[
	C(n, d) := |\N^n_d| \cdot \max_{\kappa \in \N^n_d} 2^{w(\kappa) / 2} \cdot 2{n} \pi^2 d^2.
\]
In view of Lemma \ref{LEM:outline}, we have thus proven Theorem \ref{THM:main}.
Finally, we can bound the constant $C(n,d)$ in two ways. On the one hand, we have:
\[
 |\N^n_d|=	{n+d\choose n} = \prod_{i=1}^n {d+i\over i}\leq (d+1)^n \text{ and } \max_{\kappa \in \N^n_d} w(\kappa) \leq n,
\]
resulting in a polynomial dependence of $C(n, d)$ on $d$ for fixed $n$.  On the other hand, we have:
\[
	|\N^n_d|={n+d\choose d}  \leq (n+1)^d \text{ and } \max_{\kappa \in \N^n_d} w(\kappa) \leq d,
\]
resulting in a polynomial dependence of $C(n, d)$ on $n$ for fixed $d$. Namely, we have:
\begin{equation}\label{eqCnd}
C(n,d)\le 2\pi^2 d^2 {n} 2^{n/2}(d+1)^n  \quad \text{ and } \quad C(n,d)\le 2\pi^2 d^2{n} 2^{d/2}  (n+1)^d.
\end{equation}

\section{Concluding remarks}
We have shown that the error of the degree {$r$} Lasserre-type bound \eqref{EQ:lasr} for the minimization of a polynomial over the hypercube $[-1,1]^n$ is of the order $O(1/r^2$) when using a sum-of-squares decomposition in the truncated preordering. Alternatively, if $f$ is a polynomial nonnegative on $[-1, 1]^n$ and $\eta > 0$, our result may be interpreted as showing a bound in $O(1/\sqrt{\eta})$ on the degree of a Schm\"udgen-type certificate of positivity for $f+\eta$. The dependence on the dimension $n$ and the degree $d$ of $f$ in the constants of our result is both polynomial in $n$ (for fixed $d$), and polynomial in $d$ (for fixed $n$).

\subsection*{{The constant $C(n, d)$}}
A question left open in this work is whether it is possible to show Theorem \ref{THM:main} with a constant $C(d)$ that only depends on the degree $d$ of $f$, and not on the number of variables $n$ (cf. \eqref{eqCnd}).
This question is motivated by the fact that for the analysis of the analogous hierarchies for the unit sphere in \cite{FangFawzi2021} and for the boolean hypercube in \cite{SlotLaurent2020b} the existence of such a constant (depending only on $d$) was in fact shown.

\subsection*{{Relation to recent developments}}
Recently, there has been growing interest in obtaining a sharper convergence analysis for various Lasserre-type hierarchies for the minimization of a polynomial $f$ over a semialgebraic set $S = \{{\x \in \R^n} : g_j(\x) \ge 0 ~~ (j\in [m]) \}$. Our work thus contributes to this research area. We outline some recent developments.

We refer to the works \cite{deKlerkLaurent2020,SlotLaurent2021b} (and further references therein) for the analysis of hierarchies of upper bounds (obtained by minimizing the expected value of $f$ on $S$ with respect to a sum-of-squares density). 

The most commonly used hierarchies of lower bounds are defined in terms of sums-of-squares decompositions in the \emph{quadratic module} of $S$, being the set of conic combinations of the form $\sigma_0+\sum_{j=1}^m \sigma_jg_j$ with $\sigma_j\in \Sigma[\x]$. {Such decompositions are called \emph{Putinar-type} certificates}. In comparison, the \emph{preordering} $Q(S)$ also involves conic combinations of the \emph{products} of the $g_j$. In \cite{NieSchweighofer} a degree bound in $O(\exp(\eta^{-c}))$ is given for the quadratic module, where $c>0$ is a constant depending on $S$. 

In a {recent} work \cite{BaldiMourrain2021b}, Baldi \& Mourrain are able to improve this result to obtain a bound with a polynomial dependency on $\eta$. Roughly speaking, their method of proof relies on embedding the semialgebraic set $S$ in a box $[-R, R]^n$ of large enough size $R>0$, and then relating positivity certificates on $S$ to those on $[-R, R]^n$. Our present result on $[-1, 1]^n$ then allows them to conclude their analysis. Their argument relies on the fact the constant $C(n, d)$ in Theorem~\ref{THM:main} may be chosen to depend polynomially on the degree $d$ of $f$. Such a dependence was not shown in the earlier work \cite{deKlerkLaurent2010}.

{Note that it has been shown in \cite{Nie2014} that the hierarchies of bounds based on Putinar type representations have {\em finite} convergence for {\em generic} problems. However, and perhaps somewhat surprisingly, their convergence analysis (for general problems) has remained a challenging problem.}

We also wish to note that a polynomial degree bound was shown already in \cite{MaiMagron-PV-2021} for a slightly different hierarchy, based on  Putinar-Vasilescu type representations, which give a decomposition in the quadratic module after multiplying the polynomial $f+\eta$ by a suitable power $(1+\sum_{i=1}^n\|\x\|^2)^k$  (under some conditions).

\subsection*{{Putinar vs. Schm\"udgen on the hypercube}}
{As mentioned, Putinar-type hierarchies (making use of the quadratic module) are more commonly applied in practice than the Schm\"udgen-type hierarchy (making use of the preordering) that we consider in this paper. It is therefore natural to consider the status of convergence results for Putinar-type hierarchies on the hypercube $\cube{n}$.}

{
Magron~\cite{Magron} shows a degree bound in $O(\exp(c \eta^{-1}))$ for Putinar-type certificates of $f + \eta$ on $\cube{n}$, improving the general result of~\cite{NieSchweighofer} in this special case\footnote{The cube $[0, 1]^n$ is considered in~\cite{Magron}, but all results carry over immediately to $[-1, 1]^n$ after an affine change of variables.}. His result relies on the degree bound in $O(\eta^{-1})$ for \emph{Schm\"udgen}-type certificates on $\cube{n}$ shown in~\cite{deKlerkLaurent2010}. Importantly, it is contingent on an unresolved conjecture also posed in~\cite{deKlerkLaurent2010}: For each $n \in \N$ even, the polynomial $2^{-n} (1-x_1)(1-x_2) \ldots (1-x_n) + \eta$ lies in the quadratic module of $B^n$ truncated at degree $n$ for $\eta = \frac{1}{n(n+2)}$. {This open question, which asks for an exact estimation of the constant that needs to be added to each generator of the preordering of $Q([-1,1]^n)$ in order to ensure membership in the quadratic module, remains interesting in itself.}
}

{
In principle, our new degree bounds for Schm\"udgen-type certificates on $\cube{n}$ could (slightly) improve  the result of Magron (which relies on the weaker bounds of~\cite{deKlerkLaurent2010}). However, such an improvement would still depend exponentially on $1/\eta$, in addition to being contingent on a conjecture. Furthermore, it seems to us that it is in any case superseded by the new result of Baldi \& Mourrain~\cite{BaldiMourrain2021b}  mentioned above, which (when specialized to the hypercube) shows degree bounds for Putinar-type certificates with \emph{polynomial} dependency on $1/\eta$.
It is an open question whether the degree bound in $O(1/\sqrt{\eta})$ we have shown here for Schm\"udgen-type certificates on $\cube{n}$ may be extended to Putinar-type certificates.
}

{
Lastly, we wish to mention that error bounds for the Putinar-type Lasserre hierarchy on the hypercube $\cube{n}$ were already provided in~\cite{Nie}. There, however, the author considers a regime where the order $r$ of the relaxation is fixed, while the dimension $n$ tends to infinity. His results are therefore not directly comparable to those of the present paper or to those discussed above.
}

\subsection*{{Negative results}}
{We have so far focused our discussion on \emph{positive} results concerning sum-of-squares representations. That is, results that give \emph{upper} bounds on the error of Lasserre's bound~\eqref{EQ:lasr}; or equivalently on the required degree of Schm\"udgen-type positivity certificates. In order to put these results in context, it would be interesting to have compl ementary \emph{negative} results, thus giving \emph{lower} bounds on the convergence rate of the Lasserre hierarchy.}

{The only applicable negative result known to the authors is due to Stengle~\cite{Stengle1996}. He considers the interval $[-1, 1] \subseteq \R$ with the semialgebraic description:
\[
[-1, 1] = \{ x \in \R : (1-x^2)^3 \geq 0\}.
\] 
Note that this description is different from the (more natural) description~\eqref{EQ:semialgdesc} that we have used in this paper. In particular, Theorem~\ref{THM:MarkovLukacz} does not apply to it. Writing $Q((1-x^2)^3)_r$ for the corresponding (truncated) preordering, Stengle shows that
\[
	1-x^2 + \eta \in Q((1-x^2)^3)_r
\]
only when $r = \Omega(1/\sqrt{\eta})$. In other words, he shows for $f(x) = 1-x^2$ that the Lasserre-type bound $\richlowbound{r}$ obtained by replacing $Q(1-x^2)_r$ in~\eqref{EQ:lasr} by $Q((1-x^2)^3)_r$ satisfies:
\[
\funcmin - \richlowbound{r} = \Omega(1/r^2).
\]
On the one hand, it is remarkable that Stengle's lower bound in $\Omega(1/r^2)$ matches the upper bound in $O(1/r^2)$ we show in this paper exactly. On the other hand, we emphasize that Stengle's result relies heavily on the nonstandard description of $[-1, 1]$ as a semialgebraic set. We leave the question of proving negative results for the standard description~\eqref{EQ:semialgdesc} for future research.
}

\subsubsection*{Acknowledgments}
We thank Lorenzo Baldi and Bernard Mourrain for their helpful suggestions. We also thank Etienne de Klerk and Felix Kirschner for useful discussions. {Furthermore, we thank Markus Schweighofer for bringing to our attention the paper of Stengle~\cite{Stengle1996}}. {We are grateful to two reviewers for their careful reading and useful suggestions; in particular we thank the referees for bringing the papers \cite{Magron,Nie} to our attention.}


\begin{thebibliography}{10}
\small

\bibitem{BaldiMourrain2021b}
L. Baldi and B. Mourrain. On the effective Putinar’s Positivstellensatz and moment approximation, \newblock {\em Mathematical Programming}, 2022. \url{https://doi.org/10.1007/s10107-022-01877-6}


\bibitem{FangFawzi2021}
K.~Fang and H.~Fawzi.
\newblock The sum-of-squares hierarchy on the sphere, and applications in
  quantum information theory.
\newblock {\em Mathematical Programming}, 2020. \url{https://doi.org/10.1007/s10107-020-01537-7}

\bibitem{HessdeKlerkLaurent2017}
E. de Klerk, R. Hess and M. Laurent.
\newblock Improved convergence rates for Lasserre-type hierarchies of upper bounds for box-constrained polynomial optimization.
\newblock {\em SIAM Journal on Optimization}, \textbf{27}(1), 347--367, 2017.

\bibitem{deKlerkLaurent2010}
E. de Klerk and M. Laurent.
\newblock Error bounds for some semidefinite programming approaches to polynomial minimization on the hypercube.
\newblock {\em SIAM Journal on Optimization,} \textbf{20}(6), 3104--3120, 2010.

\bibitem{deKlerkLaurent2020}
E. de Klerk and M. Laurent.
Worst-case examples for Lasserre's measure--based hierarchy for polynomial optimization on the hypercube.
{\em Mathematics of Operations Research,} \textbf{45}(1), 86--98, 2020. 

\bibitem{Lasserre2001}
J.B. Lasserre.
\newblock Global optimization with polynomials and the problem of moments.
\newblock {\em SIAM Journal on Optimization}, \textbf{11}(3), 796--817, 2001.

\bibitem{Magron}
{V. Magron.
Error bounds for polynomial optimization over the hypercube using Putinar type representations.  	\url{arXiv:1404.6145}, 2014.}

\bibitem{MaiMagron-PV-2021}
N.H.A. Mai and V. Magron. 
\newblock On the complexity of Putinar-Vasilescu's Positivstellensatz.
{\em Journal of Complexity}, \url{https://doi.org/10.1016/j.jco.2022.101663}.

\bibitem{Nie}
{J. Nie.
An approximation bound analysis for Lasserre's relaxation.
{\em Journal of the Operations Research Society of China}, \textbf{1}(3), 313--332, 2013.}

\bibitem{Nie2014}
J. Nie.
Optimality conditions and finite convergence of Lasserre's hierarchy.
{\em Mathematical programming}, \textbf{146}(1), 97--121, 2014.

\bibitem{NieSchweighofer}
J. Nie and M. Schweighofer.
On the complexity of Putinar's Positivstellensatz.
{\em Journal of Complexity},
\textbf{23}(1), 135--150, 2007.

\bibitem{Polya}
G. P\'olya.
Uber positive Darstellung von Polynomen.
{\em  Vierteljahresschrift der
Naturforschenden Gesellschaft in Z\"urich}, \textbf{73}, 14--145, 1928. 
Reprinted in:
Collected Papers, Volume 2, 309--313, Cambridge, MIT Press, 1974.

\bibitem{PR2}
V. Powers and B. Reznick. 
Polynomials that are positive on an interval. 
{\em Trans. Amer. Math. Soc.}, \textbf{352}, 4677--4692, 2000.

\bibitem{PR}
V. Powers and B. Reznick.
A new bound for P\'olya's Theorem with applications to polynomials positive on polyhedra.  
{\em J. Pure Appl. Algebra}, \textbf{164}(1--2), 221--229, 2001.

\bibitem{SlotLaurent2021b}
L.~Slot and M.~Laurent. 
Near-optimal analysis of univariate moment bounds for polynomial optimization. 
{\em Mathematical Programming}, \textbf{188}, 443--460, 2021.

\bibitem{SlotLaurent2020b}
L.~Slot and M.~Laurent. 
\newblock Sum-of-squares hierarchies for binary polynomial optimization. 
\newblock {\em Singh M., Williamson D.P. (eds) Integer Programming and Combinatorial Optimization (IPCO 2021). Lecture Notes in Computer Science}, vol 12707, pages 43--57. Springer, Cham. 

\bibitem{Stengle1996}
{
G.~Stengle.
\newblock Complexity Estimates for the Schm\"udgen Positivstellensatz.
\newblock {\em Journal of Complexity} \textbf{12}, 167--174, 1996.
}

\bibitem{Schmudgen1991}
K. Schm\"udgen.
\newblock The K-moment problem for compact semi-algebraic sets.
\newblock {\em Math. Ann.} \textbf{289}(2), 203--206, 1991.

\bibitem{Schweighofer2004}
M. Schweighofer.
\newblock On the complexity of Schmüdgen's Positivstellensatz.
\newblock {\em Journal of Complexity}, \textbf{20}(4), 529--543, 2004.

\bibitem{Szego1959}
G.~Szeg{\"o}.
\newblock {Orthogonal Polynomials}.
\newblock vol. 23 in {\em American Mathematical Society colloquium
  publications}. American Mathematical Society, 1959.

\bibitem{KernelPolynomialMethodSurvey}
A. Weiße, G. Wellein, A. Alvermann, H. Fehske.
\newblock The kernel polynomial method.
\newblock {\em Rev. Mod. Phys.} \textbf{78}, 275, 2006.

\end{thebibliography}
\end{document}